\documentclass[10pt]{amsart}
\usepackage{amsmath}
\usepackage{amscd}
\usepackage{amssymb}
\usepackage{amsthm}
\usepackage[shortlabels]{enumitem}
\usepackage{mathtools}
\usepackage{stmaryrd}
\usepackage{hyperref}

\newtheorem{theorem}{Theorem}[section]

\newtheorem{proposition}[theorem]{Proposition}
\newtheorem{definition}[theorem]{Definition}

\newtheorem{conj}[theorem]{Conjecture}

\theoremstyle{definition}

\newcommand{\N}{\mathbb{N}}

\newcommand{\cN}{{\mathcal N}}

\title{Strongly 1-bounded inner amenable groups}

\author[B. Hayes]{Ben Hayes}

\address{\parbox{\linewidth}{Department of Mathematics, University of Virginia\\
141 Cabell Drive, Kerchof Hall
P.O. Box 400137,
Charlottesville, VA 22904}}
\email{brh5c@virginia.edu}
\urladdr{https://sites.google.com/site/benhayeshomepage/}

\author[S. Kunnawalkam Elayavalli]{Srivatsav Kunnawalkam Elayavalli}
\address{\parbox{\linewidth}{Department of Mathematics, University of Maryland, College Park, \\
		4176 Campus Dr, College Park, MD 20742}}
\email{sriva@umd.edu}
\urladdr{https://sites.google.com/view/srivatsavke}

\begin{document}

\begin{abstract}
It is shown that finitely presented icc inner amenable groups yield strongly 1-bounded II$_1$ factors.
\end{abstract}

	\maketitle

    \section{Introduction}

Building on Voiculescu's free entropy theory \cite{PicuSurvey}, Jung introduced the influential notion of strong 1-boundedness \cite{Jung2007} with several applications to the study of II$_1$ factors. A striking feature of this property is the ability to account for amenable subalgebras in rigidity arguments. A testament to this is the resolution of the Peterson-Thom conjecture by the first author \cite{hayespt} using his 1-bounded entropy techniques \cite{Hayes2018} in conjunction with strong convergence arguments in random matrix theory \cite{belinschi2022strong, bordenave2023norm, CGVHIIStrong, MdLSstrongasymptoticfreenesshaar, Parraud2024strong}. Precisely this resolution locates the amenable subalgebras of interpolated free group factors as exactly those with zero $1$-bounded entropy \cite{UselessResolutionOfPT}. The purpose of this short note is to emphasize that the more general phenomena of inner amenability \cite{Ef75} also fits in well with the framework of strong 1-boundedness. This was raised as a question by Y. Ueda to the second author during the 2019 BIRS workshop ``Classification problems in von Neumann algebras''.   

\begin{theorem}
Let $G$ be a finitely presented icc inner amenable group. Then $L(G)$ is strongly 1-bounded.
\end{theorem}

The proof of the above result is short combining various existing results, including those of \cite{Hayes2018, ElekSzaboDeterminant, TuckerMeans, Shl2015}. Our main motivation to release this note is that we find it rather \emph{cute} that our argument splits into two cases depending on whether the group is sofic or not. It remains open whether all countable inner amenable groups $G$ satisfy $h(L(G))\leq 0$. We point out that the following is currently known about the von Neumann algebras associated to inner amenable groups $G$: $L(G)$ is not properly proximal \cite{DKEPfinal}; if $G$ is additionally nonamenable $L(G)$ is $\{\alpha_t\}$-rigid \cite{DrimbeComp}. However to our knowledge, it remains open whether $L(G)$ is $L^2$-rigid in the sense of Peterson \cite{PetersonDeriva}. We also define the notion of an inner amenable $C^*$-probability space, and conjecture that the GNS closures ought to be strongly 1-bounded.

\subsection*{Acknowledgements} We thank Robin Tucker-Drob for an insightful suggestion.

\section{Proof of the main result}

Recall the following: A countable discrete group $G$ is called inner amenable if the action of $G$ on the set $X=G\setminus \{1\}$ by conjugation admits an invariant mean. A countable discrete group $G$ is called strongly 1-bounded if the group von Neumann algebra $L(G)$ is strongly 1-bounded in the sense of Jung \cite{Jung2007}. We will crucially need the following result:

\begin{proposition}[Consequence of Proposition 5.1 in \cite{ElekSzaboDeterminant}] \label{elek}
If $G$ is inner amenable with conjugation invariant mean $\mu$, and if $N= \{g\in G: \ \mu(X_g)=1\}$ is trivial (where $X_g$ is the centralizer of $g$, and $\mu$ is the conjugation invariant mean), then $G$ is sofic. 
\end{proposition}

\begin{theorem}
If $G$ is inner amenable group which is either finitely presented and sofic or i.c.c. and not sofic, then $G$ is strongly 1-bounded.
\end{theorem}

\begin{proof}
If $G$ is sofic and finitely presented, then the result follows from combining \cite[Corollary 6]{TuckerMeans} with the main result of \cite{Shl2015} (see also \cite{JungL2B, vanishingl2bettis1b}). So suppose $G$ is i.c.c. and not sofic. Consider $G$ acting on $X=G\setminus \{1\}$ by conjugation and set $\mu$ to be the conjugation invariant mean. Define the set $N= \{g\in G: \ \mu(X_g)=1\}$ where $X_g$ is the centralizer of $g$. Since $G$ is not sofic, we have by Proposition \ref{elek} that $N\ne \{1\}$. Since $G$ is i.c.c., it follows that $N$ is infinite. For any finite set $F_n= \{h_1,\hdots, h_n\}\subset N$, observe that there exists an element $g_n\in G$ such that $g_nh_i=h_ig_n$ for all $1\leq i\leq n$. Indeed, $\mu(X_{h_i})=1$ for all $1\leq i\leq n$, and so $\mu(\cap_{i=1}^{n} X_{h_i})=1$. Therefore $\cap_{i=1}^{n} X_{h_i}\neq \emptyset$. Now, let $N= \{h_1, h_2, \hdots\}$, and consider $N= \bigcup_{i=1}^{\infty} F_i$, where $F_i= \{h_1,\hdots h_i\}$. By the previous procedure, we have a sequence $\{g_i\}_{i=1}^{\infty}$ of non-identity group elements that eventually commutes with all of $N$. Hence, for $\omega\in \beta(\N)\setminus \N$, we have $L(N)'\cap L(G)^\omega$ is diffuse. By Proposition 4.5 of \cite{Hayes2018}, we see $h(L(N): L(G))= h(L(N):L(G)^\omega)\leq 0$. But, it is an easy observation to see that $N$ is a normal subgroup in $G$. Hence, by Theorem 1.3 of \cite{Hayes2018}, $h(L(G))= h(W^{*}(\cN_{L(G)}(L(N))): L(G))= h(L(N):L(G))\leq 0$. By Proposition A.16 of \cite{Hayes2018}, we conclude that G is strongly 1-bounded.
\end{proof}

\begin{conj}
All inner amenable groups are strongly 1-bounded.
\end{conj}

While inner amenability of a group is not known to be an invariant of its generated von Neumann algebra, it is not hard to show that it is an invariant of its reduced group $C^{*}$-algebra equipped with the canonical trace.

\begin{definition}
  Let $(A,\tau)$ be a tracial $C^*$-probability space, i.e, a unital $C^{*}$-algebra $A$ equipped with a faithful tracial state $\tau$. We say that $(A,\tau)$ is \emph{inner amenable} if there is a sequence $(\xi_{n})_{n}$ in $L^{2}(A,\tau)$ with
\begin{itemize}
    \item $\|\xi_{n}\|_{2}=1,$
    \item $\tau(\xi_{n})=0,$
    \item $\|a\xi_{n}-\xi_{n}a\|_{2}\to_{n\to\infty} 0,$ for all $a\in A$,
\end{itemize}  
\end{definition}

Observe that a group $G$ is inner amenable if and only if $(C^{*}_{\lambda}(G),\tau)$ is inner amenable where $\tau(x)=\langle{x\delta_{1},\delta_{1}\rangle}.$

\begin{conj}
Let $(A,\tau)$ be an inner amenable tracial $C^{*}$-probability space. Then $h(W^{*}(A,\tau))\leq 0.$
\end{conj}

\bibliographystyle{alpha}
	\bibliography{references}	

\end{document}